\documentclass[11pt,a4paper,reqno]{amsart}

\usepackage{amsfonts}
\usepackage{graphicx}
\usepackage{graphics}
\usepackage{anysize}
\usepackage{wrapfig}
\usepackage{float}
\usepackage{placeins}
\marginsize{3cm}{3cm}{3cm}{3cm}
\newtheorem{theorem}{Theorem}[section]

\newtheorem{lemma}{Lemma}[section]

\newtheorem{corollary}{Corollary}[section]
\newtheorem{definition}{Definition}[section]
\newtheorem{example}{Example}[section]

\makeatletter

\begin{document}

\begin{flushleft}
\end{flushleft}
\setcounter{page}{1}
\vspace{2cm}

\title {The Diametral Metric Dimension of Generalized Corona Graph  \\ \footnotesize  }
\author{A. N. Fitria$^1$, L. Susilowati$^{1*}$, Y. Wahyuni$^1$, N. H. Sarmin$^2$, \,\S }
\thanks{\noindent$^1$ Department of Mathematics, Faculty of Science and Technology, Universitas Airlangga, Surabaya, Indonesia\\
\indent \,\,\, e-mail: anis.nur.fitria-2025@fst.unair.ac.id; ORCID: https://orcid.org/0009-0009-7394-2860.\\
\indent \,\,\, e-mail: liliek-s@fst.unair.ac.id; ORCID: https://orcid.org/0000-0002-9149-3570.\\
\indent \,\,\, e-mail: yayuk-w@fst.unair.ac.id; ORCID: https://orcid.org/0009-0009-8782-243X.\\
\indent$^2$ Department of Mathematical Sciences, Faculty of Science, Universiti Teknologi Malaysia, Johor Bahru, Malaysia\\
\indent \,\,\, e-mail: nhs@utm.my; ORCID: https://orcid.org/0000-0003-4291-5746.\\
\indent  $^*$\,\,Corresponding author\\ 
\indent \,\,\, The research was supported by Airlangga Research Fund (ARF) year 2026 Contract No. 1819/B/DST/UN3.DRI/PT.01.03/2026.}

\begin{abstract}
This study investigates the diametral metric dimension of generalized corona graphs, where the main graph is connected graph and the branch graphs form a sequence of connected graphs. The concept of diametral metric dimension is an extension of the metric dimension concept, requiring the resolving set to contain all diametral vertices. To determine the diametral metric dimension of a generalized corona graph, one must first identify the distance of each vertex in the graph and the graph's diametral set. Subsequently, a resolving set containing the diametral set is determined. The results show that the diametral metric dimension of the generalized corona graph depends on the diameter of the main graph and the metric dimensions of the graphs in the sequence. These concepts provide theoretical insights into resolving structures in the planning infrastructure and motivate further studies on other graph families and graph operations.
 
\end{abstract}
\maketitle 
\bigskip
\bigskip

\section{Introduction}
Graph theory is a branch of discrete mathematics that studies mathematical structures consisting of vertices and edges, where edges represent relationships between pairs of vertices. Since the work of Leonhard Euler on the Seven Bridges of K\"onigsberg problem, graph theory has developed into an important area of mathematical research with applications in transportation, communication networks, optimization, and network design \cite{chartrand2012}.

Among the fundamental distance-based concepts in graph theory are distance, eccentricity, diameter, and radius. For two vertices $v,w\in V(G)$, the distance between $v$ and $w$, denoted by $d_G(v,w)$ or simply by $d(v,w)$, is defined as the length of a shortest path between $v$ and $w$ in $G$. Based on the concept of distance, the eccentricity of a vertex, as well as the diameter and radius of a graph, can be defined. These concepts have led to the development of various graph parameters, one of which is the metric dimension. The metric dimension was introduced independently by Slater in 1975 and by Harary and Melter in 1976 \cite{slater1975,harary1976}. Since then, metric dimension has been extensively studied for various graph classes and graph operations.

The significance of this study lies in the development of a concept in graph theory that not only enriches the study of metric dimension but also opens opportunities for further research in network structure analysis. The development of the diametral metric dimension is expected to broaden the theoretical foundation for analyzing network structures, particularly in identifying important vertices based on the distance structure and diameter of a graph.

A related direction in the study of the metric dimension is obtained by requiring a resolving set that contains the central set, this is referred to as the central metric dimension \cite{rodriguez2020}. Susilowati et al. have investigated several variants of metric dimension, including rooted product graphs and edge coronation graphs \cite{susilowati2025a,susilowati2025b}. Similarly, the diametral metric dimension requires a resolving set to contain all diametral vertices, namely the vertices whose eccentricity is equal to the diameter of the graph. 

Graph operations provide an object for investigating how structural modifications affect distance-based parameters. One such operation is the corona product, introduced by Frucht and Harary \cite{frucht1970}. The generalized corona extends this operation by allowing different graphs to be attached to different vertices of the main graph \cite{rodriguez2015}. Consequently, the resulting distance structure, eccentricities, diameter, and diametral vertices may depend on both the main graph and the attached graphs. This makes the generalized corona an interesting graph operation for studying the diametral metric dimension.

Motivated by these observations, this paper investigates the diametral metric dimension of generalized corona graphs $G\odot(H^1,H^2,\ldots,H^n)$. The main graph $G$ is considered from four fundamental graph families, namely the path, cycle, complete, and star graphs. For each graph family, we characterize the diametral vertices of the resulting generalized corona graph and determine its diametral metric dimension. The results provide further insight into the relationship between graph operations, the location of diametral vertices, and resolving sets.

\section{Preliminaries}

Throughout this paper, all graphs are assumed to be connected. Let $G=(V(G),E(G))$ be a graph. The distance between two vertices $u,v\in V(G)$, denoted by $d_G(u,v)$, is the length of a shortest $u$--$v$ path in $G$.

For a vertex $v\in V(G)$, the eccentricity of $v$ is defined by
\begin{equation}
\operatorname{e}_G(v)=\max_{u\in V(G)}d_G(u,v).
\end{equation}
The diameter and radius of $G$ are respectively defined by
\begin{equation}
\operatorname{diam}(G)=\max_{v\in V(G)}\operatorname{e}_G(v)
\end{equation}
and
\begin{equation}
\operatorname{rad}(G)=\min_{v\in V(G)}\operatorname{e}_G(v).
\end{equation}

The diameters of several fundamental graph families are well known.
\begin{lemma} \cite{chartrand2000,harary1969}
Let $G$ be a connected graph of order $n$. The following properties hold:
\begin{enumerate}
    \item $\operatorname {diam}(P_n)=n-1$;
    
    \item $\operatorname{diam}(C_n)=\left\lfloor\frac{n}{2}\right\rfloor$;
    
    \item $\operatorname{diam}(K_n)=1$;
    
    \item $\operatorname{diam}(S_n)=2$.
\end{enumerate}
\end{lemma}

\begin{definition} \cite{marino2015}
Diametral vertices are vertices whose eccentricity is equal to the diameter of the graph.
\end{definition}

\begin{definition} \cite{chartrand2000}
Let $W=\{w_1,w_2,\ldots,w_k\}\subseteq V(G)$ be an ordered set. The representation of a vertex $v\in V(G)$ with respect to $W$ is defined by
\begin{equation}
r(v\mid W)=
\left(d_G(v,w_1),d_G(v,w_2),\ldots,d_G(v,w_k)\right).
\end{equation} An ordered set $W\subseteq V(G)$ is called a resolving set of $G$ if $r(u\mid W)\neq r(v\mid W)$ for every pair of distinct vertices $u,v\in V(G)$. A resolving set of minimum cardinality is called a metric basis, and its cardinality is called the metric dimension of $G$, denoted by $\dim(G)$.
\end{definition}

The following lemma is useful in determining resolving sets.

\begin{lemma} \label{lemma:reprecase1} \cite{susilowati2015}
Let $G$ be a connected graph and let $W\subseteq V(G)$. For every $v_i,v_j\in W$ with $i\neq j$, $r(v_i\mid W)\neq r(v_j\mid W)$.
\end{lemma}

\begin{lemma} \label{lemma:proofminimumitydom} \cite{susilowati2020}
Let $G$ be a connected graph. If there is no dominant resolving set
of $G$ with cardinality $k$, then any set $W\subseteq V(G)$ with
$|W|<k$ is not a dominant resolving set.
\end{lemma}

The following result gives the metric dimension of $K_1+C_n$ and $K_1+P_n$.

\begin{lemma} \label{dimwheel-fan} \cite{javaid2008}
\begin{enumerate}
    \item[(1)] For $n\notin\{3,6\}$,  $dim(K_1+C_n)=\left\lfloor\frac{2n+2}{5}\right\rfloor;$

    \item[(2)] For $n\notin\{1,2,3,6\}$,
     $dim(K_1+P_n)=\left\lfloor\frac{2n+2}{5}\right\rfloor.$
\end{enumerate}
\end{lemma}
We next recall the generalized corona operation.

\begin{definition} \cite{rodriguez2015}
Let $G$ be a graph of order $n$ with $V(G)=\{v_1,v_2,\ldots,v_n\}$ and let $\mathcal{H}=(H^1,H^2,\ldots,H^n)$ be a sequence of graphs. The generalized corona graph, denoted by
\begin{equation}
G\odot\mathcal{H}
=
G\odot(H^1,H^2,\ldots,H^n),
\end{equation}
is the graph obtained by taking one copy of $G$ and one copy of each graph $H^i$, and joining every vertex $v_i$ of $G$ to every vertex of $H_i$ for $i=1,2,\ldots,n$.
\end{definition}

For the generalized corona graph, the vertex labeling is given by $G\odot\mathcal{H}=V(G^0)\cup\bigcup_{i=1}^{n}V(H_i^i)$, where $V(G^0)=\{v_i^0\in V(G\odot H)\mid v_i\in V(G)\}$ and $V(H_i^i)=\{u_{ij}^i\mid u_{ij}\in V(H^i),\;i=1,2,3,\ldots,n\}$. Here, $G^0$ is referred to as the main graph, while each $H^i$ is referred to as a branch graph.

\section{Main Results} 

In this section, we determine the diametral metric dimension of the generalized corona graph $G\odot\mathcal{H}$ where $G$ is a connected graph of order $n$ and $\mathcal{H}=(H^1,H^2,\ldots,H^n)$ is a sequence of connected graphs. Throughout this section, $G^0$ denotes the copy of the main graph $G$ in $G\odot\mathcal{H}$, while $H_i^i$ denotes the copy of $H^i$ attached to $v_i\in V(G)$.

The first step in this study is to define the concept of the diametral metric dimension.

\begin{definition}
A diametral set is a set whose elements are all diametral vertices. The diametral set of a graph $G$ is denoted by $D(G)$.
\end{definition}

\begin{definition}
Let $G$ be a connected graph. An ordered set $W\subseteq V(G)$ is called a diametral resolving set of $G$ if $W$ is a resolving set that also contains the diametral set $D(G)$. A diametral resolving set of minimum cardinality is called a diametral basis. The cardinality of a diametral basis in $G$ is called the diametral metric dimension, denoted by $\dim_{\mathrm{diam}}(G)$.
\end{definition}

The following lemma is adapted from the argument used for dominant resolving sets in Lemma~\ref{lemma:proofminimumitydom}.

\begin{lemma} \label{lemma:minimum}
Let $G$ be a connected graph. If $\dim_{\mathrm{diam}}(G)=k$, then every set $W\subseteq V(G)$ with $|W|<k$ is not a diametral resolving set.
\end{lemma}

\begin{proof}
By definition, $\dim_{\mathrm{diam}}(G)$ is the minimum cardinality of a diametral resolving set of $G$. Therefore, no set $W\subseteq V(G)$ with $|W|<\dim_{\mathrm{diam}}(G)=k$ can be a diametral resolving set.
\end{proof}

The next lemma presents the diametral set of the corona operation $G\odot\mathcal{H}$ for any connected graph $G$ and a sequence of connected graphs $\mathcal{H}$.

\begin{lemma} 
Let $G$ be connected graph of order $n>1$ and $\mathcal{H}=(H^1,H^2,\ldots,H^n)$ be a sequence of graphs, then
\begin{equation}
D(G\odot\mathcal{H})=\{u_{ij}^i \in V(H_i^i)\mid v_i\in D(G)\}.
\end{equation}
\end{lemma}

\begin{proof}
Suppose $G$ is a connected graph with $V(G)=\{v_i\mid i=1,2,3,\ldots,n\}$. Suppose $\mathcal{H}$ is a sequence of $n$ connected graphs, namely $\mathcal{H}=(H^1,H^2,\ldots,H^n)$ with $V(H^i)=\{u_{ij}\mid j=1,2,3,\ldots,|V(H^i)|\}$, $i=1,2,3,\ldots,n$. The vertex set of the graph $G\odot\mathcal{H}$ is $V(G\odot\mathcal{H})=V(G^0)\cup_{i=1}^{n}V(H_i^i)$ with $V(G^0)=\{v_i^0\in V(G\odot\mathcal{H})\mid v_i\in V(G)\}$ and $V(H_i^i)=\{u_{ij}^i\mid u_{ij}\in V(H^i);\,i=1,2,3,\ldots,n\}$. Suppose $v_i\in D(G)$, then $e_G(v_i)=\max\{e_G(v)\mid v\in V(G)\}$. Suppose $v_i^0\in V(G\odot\mathcal{H})$ for $v_i\in D(G)$, then for every $u_{ij}^i\in V(H_i^i)$, it follows that $e_{G\odot\mathcal{H}}(u_{ij}^i)=e_{G\odot\mathcal{H}}(v_i^0)+2$. It is known that $e_{G\odot\mathcal{H}}(u)\leq e_{G\odot\mathcal{H}}(u_{ij}^i)$, for every $u\in V(G\odot\mathcal{H})$. Therefore, $e_{G\odot\mathcal{H}}(u_{ij}^i)=\max\{e_{G\odot\mathcal{H}}(u)\mid u\in V(G\odot\mathcal{H})\}.$ Thus, $D(G\odot\mathcal{H})=\{u_{ij}^i\in V(H_i^i)\mid v_i\in D(G)\}.$
\end{proof}

\begin{lemma}
\label{lemma:representation}
Let $G$ be a connected graph and $W\subseteq V(G)$. If $x\in W$ or $y\in W$ then $r(x\mid W)\neq r(y\mid W)$.
\end{lemma}

\begin{proof}
Let $W$ be a diametral resolving set of graph $G$. Take any $x\in W$ or $y\in W$, then there are two possibilities, namely (1) $x,y\in W$; (2) $x\in W$ and $y\notin W$.
\begin{enumerate}
    \item $x,y\in W$, based on Lemma~\ref{lemma:reprecase1}, it follows that $r(x\mid W)\neq r(y\mid W).$
    \item Suppose, without loss of generality, that $x\in W$ and $y\notin W$. Then, in $r(x\mid W)$ there exists an element $0$ corresponding to $x$, whereas in $r(y\mid W)$ there is no element $0$ since $y\notin W$. Thus, $ r(x\mid W)\neq r(y\mid W).$
\end{enumerate}
Therefore, it is proved that for $x\in W$ or $y\in W$, $r(x\mid W)\neq r(y\mid W)$.
\end{proof}

We first consider the case where the main graph is a path.

\subsection{$G$ is a Path Graphs}

\begin{theorem}
Let $P_n$ be a path graph with $n\geq 3$, and let $\mathcal{H}=(H^1,H^2,\ldots,H^n)$ be a sequence of star graphs or complete graphs. Then
\begin{equation}
\dim_{\mathrm{diam}}
\left(P_n\odot\mathcal{H}\right)
=
\left|D\left(P_n\odot\mathcal{H}\right)\right|
+
\sum_{i=2}^{n-1}\dim(H^i).    
\end{equation}
\end{theorem}

\begin{proof}
Let $P_n$ be a path graph with $V(P_n)=\{v_i\mid i=1,2,3,\ldots,n\}$ and $\mathcal{H}$ be a sequence of $n$ star graphs or complete graphs, namely $\mathcal{H}=(H^1,H^2,\ldots,H^n)$ with $V(H^i)=\{u_{ij}\mid j=1,2,3,\ldots,|V(H^i)|\},\qquad i=1,2,3,\ldots,n.$
The vertex set of the graph $P_n\odot\mathcal{H}$ is given by $V(P_n\odot\mathcal{H})=V(P_n^0)\cup\bigcup_{i=1}^{n}V(H_i^i),$ where $V(P_n^0) = \{v_i^0\in V(P_n\odot\mathcal{H})\mid v_i\in V(P_n)\}$ and $V(H_i^i) = \{u_{ij}^i\mid u_{ij}\in V(H^i)\}, \qquad i=1,2,3,\ldots,n.$
Let $B^i$ be a basis of graph $H^i$, so that for $i=1,2,3,\ldots,n$,
\[
B^i=
\begin{cases}
\{u_{i2},u_{i3},\ldots,u_{i(m-2)},u_{i(m-1)}\},
& H^i\cong S_m,\\[4pt]
\{u_{i1},u_{i2},\ldots,u_{i(q-2)},u_{i(q-1)}\},
& H^i\cong K_q.
\end{cases}
\]
Let $B_i^i = \{u_{ij}^i\in V(P_n\odot\mathcal{H})\mid u_{ij}\in B^i\},\qquad i=2,3,\ldots,n-1.$
We choose $W=\{u_{ij}^i\in V(H_i^i)\mid v_i\in D(P_n)\}\cup\bigcup_{i=2}^{n-1}\{u_{ij}^i\in V(P_n\odot\mathcal{H})\mid u_{ij}\in B^i\}.$
Thus, $|W|=|D(P_n\odot\mathcal{H})|+\sum_{i=2}^{n-1}\dim(H^i).$
For any distinct vertices $u,v\in V(P_n\odot\mathcal{H})$ where $u\neq v$, there are three possible cases:
\begin{enumerate}
    \item $u,v\in W$;
    \item $u\in W$ and $v\in V(P_n\odot\mathcal{H})\setminus W$;
    \item $u,v\in V(P_n\odot\mathcal{H})\setminus W$.
\end{enumerate}
For cases (1) and (2), by Lemma~\ref{lemma:representation}, it is proven that $r(x\mid W)\neq r(y\mid W).$ For case (3), there are four subcases.

\noindent\textbf{\textit{Subcase 3.1:}}
$u_{ij}^i,u_{ik}^i\in V(H_i^i)$.

Since $r(u_{ij}^i\mid B_i^i) \neq r(u_{ik}^i\mid B_i^i)$ and $B_i^i\subseteq W$, then $r(u_{ij}^i\mid W) \neq r(u_{ik}^i\mid W).$

\noindent\textbf{\textit{Subcase 3.2:}}
$v_x^0,v_y^0\in V(P_n^0)$ where $x\neq y$.

Since $v_x^0,v_y^0\in V(P_n^0)$ is a path, $d(v_x^0,v_y^0)=s,\qquad 1\leq s\leq n-1.$ Suppose $u_{xj}^x\in W$. We know that $d(v_x^0,u_{xj}^x)=1,$ for $j=1,2,\ldots,|V(H^i)|$, so $d(v_y^0,u_{xj}^x) = d(v_x^0,v_y^0)+d(v_x^0,u_{xj}^x) = s+1.$ Because $d(v_y^0,u_{xj}^x)>d(v_x^0,u_{xj}^x),$ it holds that $r(v_x^0\mid W)\neq r(v_y^0\mid W).$

\noindent\textbf{\textit{Subcase 3.3:}}
$v_x^0\in V(P_n^0)$ and $u_{ij}^i\in V(H_i^i)$.

Let $u_{ik}^i\in W$, so $d(u_{ij}^i,v_i^0)=d(u_{ik}^i,v_i^0)=1$ and $d(u_{ij}^i,u_{ik}^i)\leq 2.$ Since $d(v_x^0,v_i^0)=s,\qquad 1\leq s\leq n-1,$ then $d(v_x^0,u_{ik}^i) = d(v_x^0,v_i^0)+d(u_{ik}^i,v_i^0) = s+1.$ Thus, $d(u_{ij}^i,u_{ik}^i) < d(v_x^0,u_{ik}^i).$ Hence, for $v_x^0,u_{ij}^i\in V(P_n\odot\mathcal{H})\setminus W,$ it holds that $r(v_x^0\mid W)\neq r(u_{ij}^i\mid W).$

\noindent\textbf{\textit{Subcase 3.4:}}
$u_{xj}^x\in V(H_x^x)$ and $u_{yj}^y\in V(H_y^y)$ for $x\neq y$.

Since $d(u_{xj}^x,v_x^0)=1, \qquad d(v_x^0,v_y^0)=s,\qquad 1\leq s\leq n-1$. Suppose $u_{yk}^y\in W$, we have $d(u_{yj}^y,u_{yk}^y)\leq 2$ and $d(v_y^0,u_{yj}^y)=d(v_y^0,u_{yk}^y)=1$. We know that $d(u_{xj}^x,v_y^0) = d(u_{xj}^x,v_x^0) + d(v_x^0,v_y^0) = 1+s$ and $d(u_{xj}^x,u_{yk}^y) = d(u_{xj}^x,v_y^0) + d(v_y^0,u_{yk}^y) = (1+s)+1 = 2+s.$ Thus, $d(u_{yj}^y,u_{yk}^y) < d(u_{xj}^x,u_{yk}^y).$ Hence, for $u_{xj}^x,u_{yj}^y \in V(P_n\odot\mathcal{H})\setminus W,$ it follows that $r(u_{xj}^x\mid W) \neq r(u_{yj}^y\mid W).$

Next, to prove that $W$ is a diametral resolving set with minimum cardinality, suppose $T\subseteq V(P_n\odot\mathcal{H})$ contains the diametral set with $|T|<|W|.$ Let $|T|=|W|-1.$ Then there is $i$ such that at most $|B_i^i|-1$ vertices in $H_i^i$ are elements of $T$. Let $S_i\subseteq$ is set with $|B_i^i|-1$ vertices Consequently, there are two vertices $u_{ij}^i,u_{ik}^i\in V(H_i^i),\qquad u_{ij}^i,u_{ik}^i\notin T,$ such that $r(u_{ij}^i\mid B_i^i) = r(u_{ik}^i\mid B_i^i).$ Moreover, for every vertex $z\in V(P_n\odot\mathcal H)\setminus V(H_i^i)$, every path from $z$ to a vertex of $H_i^i$ must pass through $v_i^0$. Since both $u_{ij}^i$ and $u_{ik}^i$ are adjacent to $v_i^0$, we have $d(z,u_{ij}^i) =d(z,v_i^0)+1=d(z,u_{ik}^i).$ Therefore, no vertex outside $H_i^i$ can distinguish $u_{ij}^i$ and $u_{ik}^i$. Consequently, $T$ is not a diametral resolving set. By Lemma~\ref{lemma:minimum}, $W$ is the diametral resolving set with minimum cardinality. 
Thus, $\dim_{\mathrm{diam}}(P_n\odot\mathcal{H}) = |D(P_n\odot\mathcal{H})| + \sum_{i=2}^{n-1}\dim(H^i).$
\end{proof}

\begin{theorem}
Let $P_n$ be a path graph with $n\geq 3$ and $\mathcal{H}=(H^1,H^2,\ldots,H^n)$ be a sequence of path graphs or cycle graphs for $|V(H^i)|>6$, then
\begin{equation}
\dim_{\mathrm{diam}}(P_n\odot\mathcal{H})
=
|D(P_n\odot\mathcal{H})|
+
\sum_{i=2}^{n-1}\dim(K_1+H^i).    
\end{equation}
\end{theorem}

\begin{proof}
Let $P_n$ be a path graph with $V(P_n)=\{v_i\mid i=1,2,3,\ldots,n\}$ and $\mathcal{H}$ be a sequence of $n$ path graphs or cycle graphs, namely $\mathcal{H}=(H^1,H^2,\ldots,H^n)$ with $V(H^i) = \{u_{ij}\mid j=1,2,3,\ldots,|V(H^i)|\}, \qquad i=1,2,3,\ldots,n.$ The vertex set of the graph $P_n\odot\mathcal{H}$ is given by $V(P_n\odot\mathcal{H}) = V(P_n^0)\cup\bigcup_{i=1}^{n}V(H_i^i),$ where $V(P_n^0) = \{v_i^0\in V(P_n\odot\mathcal{H})\mid v_i\in V(P_n)\}$ and $V(H_i^i) = \{u_{ij}^i\mid u_{ij}\in V(H^i)\}, \qquad i=1,2,3,\ldots,n.$
Since each $H^i$ is either a path graph or a cycle graph, by Lemma~\ref{dimwheel-fan}, a metric basis of $K_1+H^i$ can be chosen with cardinality $\dim(K_1+H^i)$, so that for $i=1,2,3,\ldots,n$, 
\[
B^i=
\begin{cases}
\{u_{i1},u_{i2},\ldots,u_{ir}\},
& H^i\cong P_m,\\[4pt]
\{u_{i1},u_{i2},\ldots,u_{is}\},
& H^i\cong C_q,
\end{cases}
\]
such that $|B^i| = \left\lfloor\frac{2m+2}{5}\right\rfloor$ or $|B^i| = \left\lfloor\frac{2q+2}{5}\right\rfloor.$ Let $B_i^i = \{u_{ij}^i\in V(P_n\odot\mathcal{H}) \mid u_{ij}\in B^i\}, \qquad i=2,3,\ldots,n-1.$ We choose $W= \{u_{ij}^i\in V(H_i^i)\mid v_i\in D(P_n)\} \cup\bigcup_{i=2}^{n-1} \{u_{ij}^i\in V(P_n\odot\mathcal{H}) \mid u_{ij}\in B^i\},$ so that $|W|=|D(P_n\odot\mathcal{H})|+\sum_{i=2}^{n-1}\dim(K_1+H^i).$
For any distinct vertices $u,v\in V(P_n\odot\mathcal{H})$ where $u\neq v$, there are three possible cases:
\begin{enumerate}
    \item $u,v\in W$;
    \item $u\in W$ and
    $v\in V(P_n\odot\mathcal{H})\setminus W$;
    \item $u,v\in V(P_n\odot\mathcal{H})\setminus W$.
\end{enumerate} For cases (1) and (2), by Lemma~\ref{lemma:representation}, it is proven that $r(x\mid W)\neq r(y\mid W).$ For case (3), there are four subcases.

\noindent\textbf{\textit{Subcase 3.1:}}
$u_{ij}^i,u_{ik}^i\in V(H_i^i)$.

Since $r(u_{ij}^i\mid B_i^i) \neq r(u_{ik}^i\mid B_i^i)$ and $B_i^i\subseteq W$, then $r(u_{ij}^i\mid W) \neq r(u_{ik}^i\mid W).$

\noindent\textbf{\textit{Subcase 3.2:}}
$v_x^0,v_y^0\in V(P_n^0)$ where $x\neq y$.

Since $v_x^0,v_y^0\in V(P_n^0)$ is a path, $d(v_x^0,v_y^0)=s,\qquad 1\leq s\leq n-1.$ Suppose $u_{xj}^x\in W$. We know that $d(v_x^0,u_{xj}^x)=1,$ for $j=1,2,\ldots,|V(H^i)|$, so $d(v_y^0,u_{xj}^x) = d(v_x^0,v_y^0) + d(v_x^0,u_{xj}^x) = s+1.$ Because $d(v_y^0,u_{xj}^x) > d(v_x^0,u_{xj}^x),$ it holds that $r(v_x^0\mid W) \neq r(v_y^0\mid W).$

\noindent\textbf{\textit{Subcase 3.3:}}
$v_x^0\in V(P_n^0)$ and $u_{ij}^i\in V(H_i^i)$.

Let $u_{ik}^i\in W$, so $d(u_{ij}^i,v_i^0) = d(u_{ik}^i,v_i^0) = 1$ and $d(u_{ij}^i,u_{ik}^i)\leq 2.$ Since $d(v_x^0,v_i^0)=s,\qquad 1\leq s\leq n-1,$ then $d(v_x^0,u_{ik}^i) = d(v_x^0,v_i^0) + d(u_{ik}^i,v_i^0) = s+1.$ Thus, $d(u_{ij}^i,u_{ik}^i) < d(v_x^0,u_{ik}^i).$ Hence, for $v_x^0,u_{ij}^i \in V(P_n\odot\mathcal{H})\setminus W,$ it holds that $r(v_x^0\mid W) \neq r(u_{ij}^i\mid W).$

\noindent\textbf{\textit{Subcase 3.4:}}
$u_{xj}^x\in V(H_x^x)$ and $u_{yj}^y\in V(H_y^y)$ for $x\neq y$.

Since $d(u_{xj}^x,v_x^0)=1, \qquad d(v_x^0,v_y^0)=s,\qquad 1\leq s\leq n-1$. Suppose $u_{yk}^y\in W$, we have $d(u_{yj}^y,u_{yk}^y)\leq 2$ and $d(v_y^0,u_{yj}^y)=d(v_y^0,u_{yk}^y)=1$. We know that $d(u_{xj}^x,v_y^0) = d(u_{xj}^x,v_x^0) + d(v_x^0,v_y^0) = 1+s$ and $d(u_{xj}^x,u_{yk}^y) = d(u_{xj}^x,v_y^0) + d(v_y^0,u_{yk}^y) = (1+s)+1 = 2+s.$ Thus, $d(u_{yj}^y,u_{yk}^y) < d(u_{xj}^x,u_{yk}^y).$ Hence, for $u_{xj}^x,u_{yj}^y \in V(P_n\odot\mathcal{H})\setminus W,$ it follows that $r(u_{xj}^x\mid W) \neq r(u_{yj}^y\mid W).$

Next, to prove that $W$ is a diametral resolving set with minimum cardinality, suppose $T\subseteq V(P_n\odot\mathcal{H})$ contains the diametral set with $|T|<|W|.$ Let $|T|=|W|-1.$ Then there is $i$ such that at most $|B_i^i|-1$ vertices in $K_1+H_i^i$ are elements of $T$. Consequently, there are two vertices $u_{ij}^i,u_{ik}^i\in V(H_i^i), \qquad u_{ij}^i,u_{ik}^i\notin T,$ such that $r(u_{ij}^i\mid B_i^i) = r(u_{ik}^i\mid B_i^i).$ Moreover, for every vertex $z\in V(P_n\odot\mathcal H)\setminus V(H_i^i)$, every path from $z$ to a vertex of $H_i^i$ must pass through $v_i^0$. Since both $u_{ij}^i$ and $u_{ik}^i$ are adjacent to $v_i^0$, we have $d(z,u_{ij}^i) =d(z,v_i^0)+1=d(z,u_{ik}^i).$ Therefore, no vertex outside $H_i^i$ can distinguish $u_{ij}^i$ and $u_{ik}^i$. Consequently, $T$ is not a diametral resolving set. By Lemma~\ref{lemma:minimum}, $W$ is the diametral resolving set with minimum cardinality. Thus, $\dim_{\mathrm{diam}}(P_n\odot\mathcal{H}) = |D(P_n\odot\mathcal{H})| + \sum_{i=2}^{n-1}\dim(K_1+H^i).$
\end{proof}

\subsection{$G$ is a Star Graphs}
\begin{theorem}
Let $S_n$ be a star graph with $n\geq 3$ and $\mathcal{H}=(H^1,H^2,\ldots,H^n)$ be a sequence of star graphs or complete graphs. Let $v_1\in S_n$ and $\deg(v_1)=n-1$, then
\begin{equation}
\dim_{\mathrm{diam}}(S_n\odot\mathcal{H}) = |D(S_n\odot\mathcal{H})| + \dim(H^1).    
\end{equation}
\end{theorem}

\begin{proof}
Let $S_n$ be a star graph with $V(S_n)=\{v_i\mid i=1,2,3,\ldots,n\}$ and $\mathcal{H}$ be a sequence of $n$ star graphs or complete graphs, namely $\mathcal{H}=(H^1,H^2,\ldots,H^n)$ with $V(H^i) = \{u_{ij}\mid j=1,2,3,\ldots,|V(H^i)|\}, \qquad i=1,2,3,\ldots,n.$ The vertex set of the graph $S_n\odot\mathcal{H}$ is given by $V(S_n\odot\mathcal{H}) = V(S_n^0)\cup\bigcup_{i=1}^{n}V(H_i^i),$ where $V(S_n^0) = \{v_i^0\in V(S_n\odot\mathcal{H})\mid v_i\in V(S_n)\}$ and $V(H_i^i) = \{u_{ij}^i\mid u_{ij}\in V(H^i)\}, \qquad i=1,2,3,\ldots,n.$ Let $B^1$ be a metric basis of $H^1$, so \[B^1=
\begin{cases}
\{u_{12},u_{13},\ldots,u_{1(m-2)},u_{1(m-1)}\},
& H^1\cong S_m,\\[4pt]
\{u_{11},u_{12},\ldots,u_{1(q-2)},u_{1(q-1)}\},
& H^1\cong K_q.
\end{cases}\]
Let $B_1^1 = \{u_{1j}^1\in V(S_n\odot\mathcal{H}) \mid u_{1j}\in B^1\}.$ We choose $W= \{u_{ij}^i\in V(H_i^i)\mid v_i\in D(S_n)\} \cup \{u_{1j}^1\in V(S_n\odot\mathcal{H}) \mid u_{1j}\in B^1\},$ so that $|W| = |D(S_n\odot\mathcal{H})| + \dim(H^1).$
For any distinct vertices $u,v\in V(S_n\odot\mathcal{H})$ where $u\neq v$, there are three possible cases:
\begin{enumerate}
    \item $u,v\in W$;
    \item $u\in W$ and
    $v\in V(S_n\odot\mathcal{H})\setminus W$;
    \item $u,v\in V(S_n\odot\mathcal{H})\setminus W$.
\end{enumerate} For cases (1) and (2), by Lemma~\ref{lemma:representation}, it is proven that $r(x\mid W)\neq r(y\mid W).$ For case (3), there are three subcases.

\noindent\textbf{\textit{Subcase 3.1:}}
$u_{1j}^1,u_{1k}^1\in V(H_1^1)$.

Since $r(u_{1j}^1\mid B_1^1) \neq r(u_{1k}^1\mid B_1^1)$ and $B_1^1\subseteq W$, then $r(u_{1j}^1\mid W) \neq r(u_{1k}^1\mid W).$

\noindent\textbf{\textit{Subcase 3.2:}}
$v_x^0,v_y^0\in V(S_n^0)$ where $x\neq y$.

Since $v_x^0,v_y^0\in V(S_n^0)$, $d(v_x^0,v_y^0)=s, \qquad s\in\{1,2\}.$ Suppose $u_{xj}^x\in W$. We know that $d(v_x^0,u_{xj}^x)=1,$ for $j=1,2,\ldots,|V(H^i)|$. Therefore, $d(v_y^0,u_{xj}^x) = d(v_x^0,v_y^0) + d(v_x^0,u_{xj}^x) = s+1.$ Because $d(v_y^0,u_{xj}^x) > d(v_x^0,u_{xj}^x),$ it holds that $r(v_x^0\mid W) \neq r(v_y^0\mid W).$

\noindent\textbf{\textit{Subcase 3.3:}}
$v_x^0\in V(S_n^0)$ and $u_{1j}^1\in V(H_1^1)$.

If $x=1$, choose any $u_{zk}^z\in V(H_z^z)\subseteq W$, for $z\neq1$. Then, $d(v_1^0,u_{zk}^z)=d(v_1^0,v_z^0)+d(v_z^0,u_{zk}^z)=1+1=2,$ Whereas $d(u_{1j}^1,u_{zk}^z) =d(u_{1j}^1,v_1^0)+d(v_1^0,v_z^0) +d(v_z^0,u_{zk}^z)=1+1+1=3.$ Hence, $d(v_1^0,u_{zk}^z)\neq d(u_{1j}^1,u_{zk}^z).$

If $x\neq 1$, choose $u_{xk}^x\in V(H_x^x)\subseteq W$. Then, $d(v_x^0,u_{xk}^x)=1,$ whereas $d(u_{1j}^1,u_{xk}^x)=d(u_{1j}^1,v_1^0)+d(v_1^0,v_x^0)+d(v_x^0,u_{xk}^x)=1+1+1=3.$
Therefore, $d(v_x^0,u_{xk}^x)\neq d(u_{1j}^1,u_{xk}^x).$
Thus, $r(v_x^0\mid W)\neq r(u_{1j}^1\mid W).$

Next, to prove that $W$ is a diametral resolving set with minimum cardinality, suppose $T\subseteq V(S_n\odot\mathcal{H})$ contains the diametral set with $|T|<|W|.$ Let $|T|=|W|-1.$ Then there is $i=1$ such that at most $|B_1^1|-1$ vertices in $H_1^1$ are elements of $T$. Consequently, there are two vertices $u_{1j}^1,u_{1k}^1\in V(H_1^1),\qquad u_{1j}^1,u_{1k}^1\notin T,$ such that $r(u_{1j}^1\mid B_1^1) = r(u_{1k}^1\mid B_1^1).$ Moreover, for every vertex $z\in V(S_n\odot\mathcal H)\setminus V(H_1^1)$, every path from $z$ to a vertex of $H_1^1$ must pass through $v_1^0$. Since both $u_{1j}^1$ and $u_{1k}^1$ are adjacent to $v_1^0$, we have $d(z,u_{1j}^1) =d(z,v_1^0)+1=d(z,u_{1k}^1).$ Therefore, no vertex outside $H_1^1$ can distinguish $u_{1j}^1$ and $u_{1k}^1$. Consequently, $T$ is not a diametral resolving set. By Lemma~\ref{lemma:minimum}, $W$ is the diametral resolving set with minimum cardinality. Thus, $\dim_{\mathrm{diam}}(S_n\odot\mathcal{H}) = |D(S_n\odot\mathcal{H})| + \dim(H^1).$
\end{proof}

\begin{theorem}
Let $S_n$ be a star graph with $n\geq 3$ and $\mathcal{H}=(H^1,H^2,\ldots,H^n)$ be a sequence of path graphs or cycle graphs for $|V(H^i)|>6$. Let $v_1\in S_n$ and $\deg(v_1)=n-1$, then
\begin{equation}
\dim_{\mathrm{diam}}(S_n\odot\mathcal{H}) = |D(S_n\odot\mathcal{H})| + \dim(K_1+H^1).    
\end{equation}
\end{theorem}

\begin{proof}
Let $S_n$ be a star graph with $V(S_n)=\{v_i\mid i=1,2,3,\ldots,n\}$ and $\mathcal{H}$ be a sequence of $n$ path graphs or cycle graphs, namely $\mathcal{H}=(H^1,H^2,\ldots,H^n)$ with $V(H^i) = \{u_{ij}\mid j=1,2,3,\ldots,|V(H^i)|\},\qquad i=1,2,3,\ldots,n.$ The vertex set of the graph $S_n\odot\mathcal{H}$ is given by $V(S_n\odot\mathcal{H}) = V(S_n^0)\cup\bigcup_{i=1}^{n}V(H_i^i),$ where $V(S_n^0) = \{v_i^0\in V(S_n\odot\mathcal{H})\mid v_i\in V(S_n)\}$ and $V(H_i^i) = \{u_{ij}^i\mid u_{ij}\in V(H^i)\}, \qquad i=1,2,3,\ldots,n.$ Since $H^1$ is either a path graph or a cycle graph, by Lemma~\ref{dimwheel-fan}, a metric basis of $K_1+H^1$ can be chosen with cardinality $\dim(K_1+H^1)$, so \[
B^1=\begin{cases}
\{u_{11},u_{12},\ldots,u_{1r}\}, & H^1\cong P_m,\\[4pt]
\{u_{11},u_{12},\ldots,u_{1s}\}, & H^1\cong C_q,
\end{cases}
\] such that $|B^1| = \left\lfloor\frac{2m+2}{5}\right\rfloor$ or $|B^1| = \left\lfloor\frac{2q+2}{5}\right\rfloor.$ Let $B_1^1 = \{u_{1j}^1\in V(S_n\odot\mathcal{H}) \mid u_{1j}\in B^1\}.$ We choose $W= \{u_{ij}^i\in V(H_i^i)\mid v_i\in D(S_n)\} \cup \{u_{1j}^1\in V(S_n\odot\mathcal{H}) \mid u_{1j}\in B^1\},$ so that $|W| = |D(S_n\odot\mathcal{H})| + \dim(K_1+H^1).$ For any distinct vertices $u,v\in V(S_n\odot\mathcal{H})$ where $u\neq v$, there are three possible cases:
\begin{enumerate}
    \item $u,v\in W$;
    \item $u\in W$ and
    $v\in V(S_n\odot\mathcal{H})\setminus W$;
    \item $u,v\in V(S_n\odot\mathcal{H})\setminus W$.
\end{enumerate} For cases (1) and (2), by Lemma~\ref{lemma:representation}, it is proven that $r(x\mid W)\neq r(y\mid W).$ For case (3), there are three subcases.

\noindent\textbf{\textit{Subcase 3.1:}}
$u_{1j}^1,u_{1k}^1\in V(H_1^1)$.

Since $r(u_{1j}^1\mid B_1^1) \neq r(u_{1k}^1\mid B_1^1)$ and $B_1^1\subseteq W$, then $r(u_{1j}^1\mid W) \neq r(u_{1k}^1\mid W).$

\noindent\textbf{\textit{Subcase 3.2:}}
$v_x^0,v_y^0\in V(S_n^0)$ where $x\neq y$.

Since $v_x^0,v_y^0\in V(S_n^0)$, $d(v_x^0,v_y^0)=s, \qquad s\in\{1,2\}.$ Suppose $u_{xj}^x\in W$. We know that $d(v_x^0,u_{xj}^x)=1.$ Therefore, $d(v_y^0,u_{xj}^x) = d(v_x^0,v_y^0) + d(v_x^0,u_{xj}^x) = s+1.$ Because $d(v_y^0,u_{xj}^x) > d(v_x^0,u_{xj}^x),$ it holds that $r(v_x^0\mid W) \neq r(v_y^0\mid W).$

\noindent\textbf{\textit{Subcase 3.3:}}
$v_x^0\in V(S_n^0)$ and $u_{1j}^1\in V(H_1^1)$.

If $x=1$, choose any $u_{zk}^z\in V(H_z^z)\subseteq W$, for $z\neq1$. Then, $d(v_1^0,u_{zk}^z)=d(v_1^0,v_z^0)+d(v_z^0,u_{zk}^z)=1+1=2,$ Whereas $d(u_{1j}^1,u_{zk}^z) =d(u_{1j}^1,v_1^0)+d(v_1^0,v_z^0) +d(v_z^0,u_{zk}^z)=1+1+1=3.$ Hence, $d(v_1^0,u_{zk}^z)\neq d(u_{1j}^1,u_{zk}^z).$

If $x\neq 1$, choose $u_{xk}^x\in V(H_x^x)\subseteq W$. Then, $d(v_x^0,u_{xk}^x)=1,$ whereas $d(u_{1j}^1,u_{xk}^x)=d(u_{1j}^1,v_1^0)+d(v_1^0,v_x^0)+d(v_x^0,u_{xk}^x)=1+1+1=3.$
Therefore, $d(v_x^0,u_{xk}^x)\neq d(u_{1j}^1,u_{xk}^x).$
Thus, $r(v_x^0\mid W)\neq r(u_{1j}^1\mid W).$

Next, to prove that $W$ is a diametral resolving set with minimum cardinality, suppose $T\subseteq V(S_n\odot\mathcal{H})$ contains the diametral set with $|T|<|W|.$ Let $|T|=|W|-1.$ Then there is $i=1$ such that at most $|B_1^1|-1$ vertices in $K_1+H_1^1$ are elements of $T$. Consequently, there are two vertices $u_{1j}^1,u_{1k}^1\in V(H_1^1), \qquad u_{1j}^1,u_{1k}^1\notin T,$ such that $r(u_{1j}^1\mid B_1^1) = r(u_{1k}^1\mid B_1^1).$ Moreover, for every vertex $z\in V(S_n\odot\mathcal H)\setminus V(H_1^1)$, every path from $z$ to a vertex of $H_1^1$ must pass through $v_1^0$. Since both $u_{1j}^1$ and $u_{1k}^1$ are adjacent to $v_1^0$, we have $d(z,u_{1j}^1) =d(z,v_1^0)+1=d(z,u_{1k}^1).$ Therefore, no vertex outside $H_1^1$ can distinguish $u_{1j}^1$ and $u_{1k}^1$. Consequently, $T$ is not a diametral resolving set. By Lemma~\ref{lemma:minimum}, $W$ is the diametral resolving set with minimum cardinality. Thus, $\dim_{\mathrm{diam}}(S_n\odot\mathcal{H}) = |D(S_n\odot\mathcal{H})| + \dim(K_1+H^1).$
\end{proof}

\subsection{$G$ is a Cycle or Complete Graph}

\begin{theorem}\label{dimdiam-cyclecomplete}
Suppose that $G$ is a cycle graph or a complete graph and $\mathcal{H}=(H^1,H^2,\ldots,H^n)$ is a sequence of graphs. Then
\begin{equation}
\dim_{\mathrm{diam}} \left(G\odot\mathcal{H}\right) = \left|D\left(G\odot\mathcal{H}\right)\right|.    
\end{equation}
\end{theorem}

\begin{proof}
Suppose $G$ is a cycle graph or a complete graph with $V(G)=\{v_i \mid i=1,2,3,\ldots,n\}$ and $H$ is a sequence of graphs, namely $\mathcal{H}=(H^1,H^2,\ldots,H^n),$ with $V(H^i)=\{u_{ij}\mid j=1,2,3,\ldots,|V(H^i)|\},\qquad i=1,2,3,\ldots,n.$ The vertex set of the graph $G\odot\mathcal{H}$ is given by $V(G\odot\mathcal{H}) = V(G^0)\cup\bigcup_{i=1}^{n}V(H_i^i),$ where $V(G^0) = \{v_i^0\in V(G\odot\mathcal{H})\mid v_i\in V(G)\}$ and $V(H_i^i) = \{u_{ij}^i\mid u_{ij}\in V(H^i)\}, \qquad i=1,2,3,\ldots,n.$ We choose $W= \{u_{ij}^i\in V(H_i^i)\mid v_i\in D(G)\}$ so that $|W|=|D(G\odot\mathcal{H})|.$ For any distinct vertices $u,v\in V(G\odot\mathcal{H})$ where $u\neq v$, there are three possible cases:
\begin{enumerate}
    \item $u,v\in W$;
    \item $u\in W$ and $v\in V(G\odot\mathcal{H})\setminus W$;
    \item $u,v\in V(G\odot\mathcal{H})\setminus W$.
\end{enumerate} For cases (1) and (2), by Lemma~\ref{lemma:representation}, it is proven that $r(x\mid W)\neq r(y\mid W).$ For case (3), let $u_{ik}^i\in W.$ Then $d(u_{ij}^i,v_i^0) = d(u_{ik}^i,v_i^0) = 1$ and $d(u_{ij}^i,u_{ik}^i)\leq 2.$ Since $d(v_x^0,v_y^0)=s, \qquad 1\leq s\leq \left\lfloor\frac{n}{2}\right\rfloor$ and $d(v_x^0,u_{xj}^x)=1 \qquad j=1,2,\ldots,|V(H^i)|,$ then $
d(v_y^0,u_{xj}^x) = d(v_x^0,v_y^0) + d(v_x^0,u_{xj}^x) =s+1.$ Because $d(v_y^0,u_{xj}^x) > d(v_x^0,u_{xj}^x),$ it follows that $r(v_x^0\mid W)\neq r(v_y^0\mid W).$
Because $W$ is a diametral set and also a resolving set, $W$ is a diametral basis. Thus, $\dim_{\mathrm{diam}}(G\odot\mathcal{H}) = |D(G\odot\mathcal{H})|.$
\end{proof}

\begin{example}
\label{ex:generalized-corona}

Figure~\ref{fig:c4} presents the graph $C_4$, while Figure~\ref{fig:sequence-H} presents the sequence of graphs $\mathcal{H}=\{P_3,C_3,K_4,S_4\}$. The vertex set of $C_4$ is denoted by $V(C_4)=\{v_i\mid i=1,2,3,4\},$ and the vertex set of each graph $H^i$ in $\mathcal{H}$ is denoted by $V(H^i)=\{u_{ij}\mid j=1,2,\ldots,|V(H^i)|\},\qquad i=1,2,3,4.$

\begin{figure}[H]
    \centering
    \includegraphics[width=0.26\textwidth]{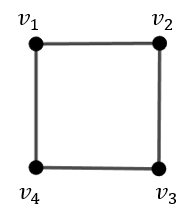}
    \caption{The graph $C_4$.}
    \label{fig:c4}
\end{figure}

\begin{figure}[H]
    \centering
    \includegraphics[width=0.96\textwidth]{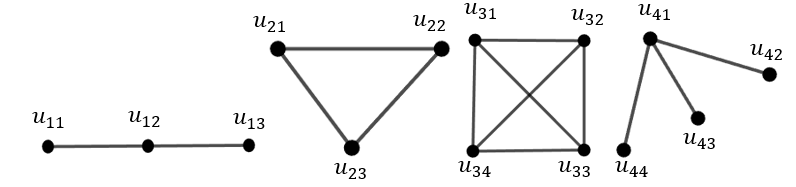}
    \caption{The sequence of graphs $\mathcal{H}=\{P_3,C_3,K_4,S_4\}$.}
    \label{fig:sequence-H}
\end{figure}

The generalized corona $C_4\odot\mathcal{H}$, is shown in Figure~\ref{fig:c4-corona}. The vertex set of $C_4\odot\mathcal{H}$ is given by $V(C_4\odot\mathcal{H}) = V(C_4^0)\cup\bigcup_{i=1}^{4}V(H_i^i),$ where $V(C_4^0) =\{v_i^0\mid v_i\in V(C_4)\},$ and $V(H_i^i) = \{u_{ij}^i\mid u_{ij}\in V(H^i),\ i=1,2,3,4\}.$

\begin{figure}[H]
    \centering
    \includegraphics[width=0.67\textwidth]{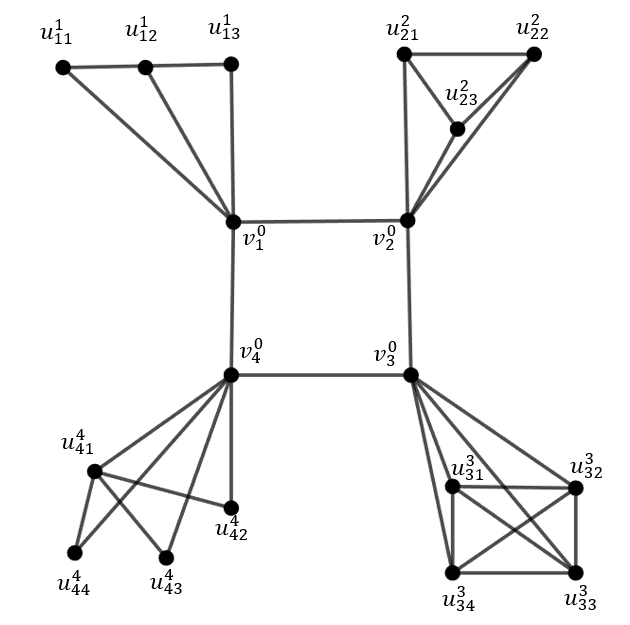}
    \caption{The generalized corona graph $C_4\odot\mathcal{H}$.}
    \label{fig:c4-corona}
\end{figure}

The diametral set of $C_4\odot\mathcal{H}$ is given by  $D(C_4\odot\mathcal{H}) = \{u_{ij}^i\in V(H_i^i)\mid i=1,2,3,4\}.$ Since $\operatorname{diam}(C_4) = \operatorname{rad}(C_4) = 2,$ it follows from Theorem~\ref{dimdiam-cyclecomplete} that $\dim_{\mathrm{diam}}(C_4\odot\mathcal{H}) = |D(C_4\odot\mathcal{H})| = 14.$ 
\end{example}

\subsection{Generalized Corona Graph}
\begin{theorem} \label{thm:3.6}
Let $G$ be a connected graph and $\mathcal{H}=(H^1,H^2,\ldots,H^n)$ be a sequence of connected graphs for $|V(H^i)|>6$, then
\begin{equation}
\dim_{\mathrm{diam}}(G\odot\mathcal{H})=
\begin{cases}
|D(G\odot\mathcal{H})|,& \operatorname{rad}(G)=\operatorname{diam}(G), \\[6pt]
|D(G\odot\mathcal{H})| +\displaystyle\sum_{\substack{i,\ v_i\notin D(G)}} \dim(H^i),& \operatorname{diam}(H^i)\leq 2, \forall i\\  
|D(G\odot\mathcal{H})| +\displaystyle\sum_{\substack{i,\ v_i\notin D(G)}} \dim(K_1+H^i), & \operatorname{diam}(H^i)> 2, \forall i.
\end{cases}    
\end{equation}
\end{theorem}

\begin{proof}
Let $G$ be a connected graph with $V(G)=\{v_i\mid i=1,2,3,\ldots,n\}$ and let $\mathcal{H}=(H^1,H^2,\ldots,H^n)$ be a sequence of $n$ connected graphs, where $V(H^i)=\{u_{ij}\mid j=1,2,3,\ldots,|V(H^i)|\}, \qquad i=1,2,3,\ldots,n.$ The vertex set of the graph $G\odot\mathcal{H}$ is given by $V(G\odot\mathcal{H}) = V(G^0)\cup\bigcup_{i=1}^{n}V(H_i^i),$ where $V(G^0)=\{v_i^0\in V(G\odot\mathcal{H})\mid v_i\in V(G)\}$ and $V(H_i^i) = \{u_{ij}^i\mid u_{ij}\in V(H^i)\}, \qquad i=1,2,3,\ldots,n.$ There are three possible cases to determine the diametral metric dimension of $G\odot\mathcal{H}$: (1) $\operatorname{diam}(G)=\operatorname{rad}(G)$; 
(2) $\operatorname{diam}(H^i)\leq2, \forall i$,; and 
(3) $\operatorname{diam}(H^i)>2, \forall i$.

\begin{enumerate}
\item[\textbf{(1)}] \textbf{$\operatorname{diam}(G)=\operatorname{rad}(G)$}.\\
We choose $W=\{u_{ij}^i\in V(H_i^i)\mid v_i\in D(G)\}$ so that $|W|=|D(G\odot\mathcal{H})|.$ For any distinct vertices $u,v\in V(G\odot\mathcal{H})$, there are three possible cases:
\begin{enumerate}
    \item[(1a)] $u,v\in W$;
    \item[(1b)] $u\in W$ and $v\in V(G\odot\mathcal{H})\setminus W$;
    \item[(1c)] $u,v\in V(G\odot\mathcal{H})\setminus W$.
\end{enumerate}
For cases (1a) and (1b), by Lemma~\ref{lemma:representation}, it is proven that $r(u\mid W)\neq r(v\mid W).$ For case (1c), let $u_{ik}^i\in W$. Then $d(u_{ij}^i,v_i^0) = d(u_{ik}^i,v_i^0) = 1$ and $d(u_{ij}^i,u_{ik}^i)\leq 2.$ Since $d(v_x^0,v_y^0)=s, \qquad 1\leq s\leq\left\lfloor\frac{n}{2}\right\rfloor,$ and $d(v_x^0,u_{xj}^x)=1,$ where $j=1,2,\ldots,|V(H^i)|$, then $ d(v_y^0,u_{xj}^x) = d(v_x^0,v_y^0)+d(v_x^0,u_{xj}^x) =s+1.$ Because $d(v_y^0,u_{xj}^x) > d(v_x^0,u_{xj}^x),$ it follows that $r(v_x^0\mid W)\neq r(v_y^0\mid W).$ Because $W$ is a diametral set and also a resolving set, $W$ is a diametral basis. Thus, $\dim_{\mathrm{diam}}(G\odot\mathcal{H}) = |D(G\odot\mathcal{H})|.$

\item[\textbf{(2)}] \textbf{$\operatorname{diam}(H^i)\leq2, \forall i$.}\\
The possibilities are $\operatorname{diam}(H^i)=1 \quad\text{or}\quad \operatorname{diam}(H^i)=2.$
\begin{enumerate}
    \item[(2a)] $\operatorname{diam}(H^i)=1$.
   
    The case $\operatorname{diam}(H^i)=1$ implies $\operatorname{rad}(H^i)=1$. Under this condition, the  diametral metric dimension follows the same logic as Case~(1).

    \item[(2b)] $\operatorname{diam}(H^i)=2$ and $\operatorname{rad}(H^i)=2$.

    When $\operatorname{rad}(H^i) = \operatorname{diam}(H^i) = 2,$ the diametral metric dimension follows the same logic as Case~(1).

    \item[(2c)] $\operatorname{diam}(H^i)=2$ and $\operatorname{rad}(H^i)=1$.

    In this case, the distance between any two vertices $u,v$ implies $ d_{H^i} (u,v) = d_{K_1+H^i} (u,v).$ Let $B^i$ be a basis of the graph $H^i$, where $ B^i=\{u_{is}\in V(G\odot\mathcal{H})  \mid 1\leq s\leq |V(H^i)|\},$ so that $ B_i^i= \{u_{is}^i\in V(G\odot\mathcal{H}) \mid u_{is}\in B^i,\ v_i\notin D(G)\},$ for $i=1,2,\ldots,n$. We choose $W=\{u_{ij}^i\in V(H_i^i)\mid v_i\in D(G)\}\cup \bigcup_{i=1}^{n} \{u_{is}^i\in V(G\odot\mathcal{H})\mid u_{is}\in B^i;\,v_i\notin D(G)\},$ so that $|W| = |D(G\odot\mathcal{H})| + \sum_{i,v_i \notin D(G)} \dim(H^i).$ For any distinct vertices $u,v\in V(G\odot\mathcal{H})$, there are three possible cases:
    \begin{enumerate}
        \item[(2c-i)] $u,v\in W$;
        \item[(2c-ii)] $u\in W$ and
        $v\in V(G\odot\mathcal{H})\setminus W$;
        \item[(2c-iii)] $u,v\in V(G\odot\mathcal{H})\setminus W$.
    \end{enumerate}

    For cases (2c-i) and (2c-ii), by Lemma~2.2, it is proven that $r(u\mid W)\neq r(v\mid W).$ For case (3c-iii), there are four possible subcases:
    \begin{enumerate}
        \item[(2c-iii-a)]
        $u_{ij}^i,u_{ik}^i\in V(H_i^i)$.

        Since $r(u_{ij}^i\mid B_i^i) \neq r(u_{ik}^i\mid B_i^i)$ and $B_i^i\subseteq W$, then $r(u_{ij}^i\mid W) \neq r(u_{ik}^i\mid W).$

        \item[(2c-iii-b)]
        $v_x^0,v_y^0\in V(G^0)$, where $x\neq y$.

        Let $d(v_x^0,v_y^0)=s, \qquad 1\leq s\leq n-1.$ Suppose that $u_{xj}^x\in W$. Then $d(v_x^0,u_{xj}^x)=1$ and $d(v_y^0,u_{xj}^x) = d(v_x^0,v_y^0)+d(v_x^0,u_{xj}^x) =s+1.$ Because $d(v_y^0,u_{xj}^x) > d(v_x^0,u_{xj}^x),$ it follows that $ r(v_x^0\mid W)\neq r(v_y^0\mid W).$

        \item[(2c-iii-c)]
        $v_x^0\in V(G^0)$ and $u_{ij}^i\in V(H_i^i)$.

        Let $u_{ik}^i\in W$, so $d(u_{ij}^i,v_i^0) = d(u_{ik}^i,v_i^0) = 1$ and $ d(u_{ij}^i,u_{ik}^i)\leq2.$ Since $ d(v_x^0,v_i^0)=s, 1\leq s\leq n-1,$ we obtain $d(v_x^0,u_{ik}^i)= d(v_x^0,v_i^0)+d(v_i^0,u_{ik}^i) =s+1.$ Hence, $ d(u_{ij}^i,u_{ik}^i)  < d(v_x^0,u_{ik}^i),$ and therefore $ r(v_x^0\mid W) \neq r(u_{ij}^i\mid W).$

        \item[(2c-iii-d)]
        $u_{xj}^x\in V(H_x^x)$ and $u_{yj}^y\in V(H_y^y)$, where $x\neq y$.

       Since $d(u_{xj}^x,v_x^0)=1, \qquad d(v_x^0,v_y^0)=s,\qquad 1\leq s\leq n-1$. Suppose $u_{yk}^y\in W$, We have $d(u_{yj}^y,u_{yk}^y)\leq 2$ and $d(v_y^0,u_{yj}^y)=d(v_y^0,u_{yj}^y)=1$. We know that $d(u_{xj}^x,v_y^0) = d(u_{xj}^x,v_x^0) + d(v_x^0,v_y^0) = 1+s$ and $d(u_{xj}^x,u_{yk}^y) = d(u_{xj}^x,v_y^0) + d(v_y^0,u_{yk}^y) = (1+s)+1 = 2+s.$ Thus, $d(u_{yj}^y,u_{yk}^y)  < d(u_{xj}^x,u_{yk}^y).$ Hence, for $u_{xj}^x,u_{yj}^y \in V(G\odot\mathcal{H})\setminus W,$ it follows that $r(u_{xj}^x\mid W) \neq r(u_{yj}^y\mid W).$
    \end{enumerate}

    Next, to prove that $W$ is a diametral resolving set with minimum cardinality, suppose $T\subseteq V(G\odot\mathcal{H})$ contains the diametral set with $|T|<|W|.$ Let $|T|=|W|-1.$ Then there is an index $i$ such that at most $|B_i^i|-1$ vertices in $H_i^i$ are elements of $T$. Consequently, there are two vertices  $u_{ij}^i,u_{ik}^i\in V(H_i^i)$, $u_{ij}^i,u_{ik}^i\notin T$, such that $r(u_{ij}^i\mid B_i^i) = r(u_{ik}^i\mid B_i^i).$ Moreover, for every vertex $z\in V(P_n\odot\mathcal H)\setminus V(H_i^i)$, every path from $z$ to a vertex of $H_i^i$ must pass through $v_i^0$. Since both $u_{ij}^i$ and $u_{ik}^i$ are adjacent to $v_i^0$, we have $d(z,u_{ij}^i) =d(z,v_i^0)+1=d(z,u_{ik}^i).$ Therefore, no vertex outside $H_i^i$ can distinguish $u_{ij}^i$ and $u_{ik}^i$. Consequently, $T$ is not a diametral resolving set. By Lemma~\ref{lemma:minimum}, $W$ is the diametral resolving set with minimum cardinality. Thus, $\dim_{\mathrm{diam}}(G\odot\mathcal{H}) = |D(G\odot\mathcal{H})| + \sum_{i,v_i\notin D(G)}\dim(H^i).$
\end{enumerate}

\item[\textbf{(3)}] \textbf{$\operatorname{diam}(H^i)>2, \forall i$.}\\
For $\operatorname{diam}(H^i)>2$, the distance between any two vertices $u,v\in V(H^i)$ is $d_{H^i}(u,v)\geq d_{G\odot\mathcal{H}}(u,v).$ Therefore, it is necessary to see $dim(K_1+H^i)$ rather that $dim(H^i)$. Let $B^i$ be a basis of the graph $K_1+H^i$ and let $B_i^i= \{u_{is}^i\in V(G\odot\mathcal{H}) \mid u_{is}\in B^i,\ v_i\notin D(G)\},$ for $i=1,2,\ldots,n$. We choose $W=\{u_{ij}^i\in V(H_i^i)\mid v_i\in D(G)\}\cup\bigcup_{i=1}^{n}\{u_{is}^i\in V(G\odot\mathcal{H})\mid u_{is}\in B^i;\,v_i\notin D(G)\},$ so that $|W|=|D(G\odot\mathcal{H})|+\sum_{i,v_i\notin D(G)}\dim(K_1+H^i).$ For any distinct vertices $u,v\in V(G\odot\mathcal{H})$, there are three possible cases:
\begin{enumerate}
    \item[(3a)] $u,v\in W$;
    \item[(3b)] $u\in W$ and $v\in V(G\odot\mathcal{H})\setminus W$;
    \item[(3c)] $u,v\in V(G\odot\mathcal{H})\setminus W$.
\end{enumerate} For cases (3a) and (3b), by Lemma~2.2, it is proven that $r(u\mid W)\neq r(v\mid W).$ For case (3c), there are four subcases:
\begin{enumerate}
    \item[(3c-i)] $u_{ij}^i,u_{ik}^i\in V(H_i^i)$.
   
    Since $r(u_{ij}^i\mid B_i^i) \neq r(u_{ik}^i\mid B_i^i)$ and $B_i^i\subseteq W$, then $r(u_{ij}^i\mid W) \neq r(u_{ik}^i\mid W).$

    \item[(3c-ii)] $v_x^0,v_y^0\in V(G^0)$, where $x\neq y$.

    Let $d(v_x^0,v_y^0)=s, \qquad 1\leq s\leq n-1.$  Suppose that $u_{xj}^x\in W$. We know that $d(v_x^0,u_{xj}^x)=1.$ Then $d(v_y^0,u_{xj}^x) = d(v_x^0,v_y^0)+d(v_x^0,u_{xj}^x) =s+1.$ Because $d(v_y^0,u_{xj}^x) >    d(v_x^0,u_{xj}^x),$ it holds that $r(v_x^0\mid W)\neq r(v_y^0\mid W).$

    \item[(3c-iii)] $v_x^0\in V(G^0)$ and $u_{ij}^i\in V(H_i^i)$.

    Let $u_{ik}^i\in W$, so $d(u_{ij}^i,v_i^0) = d(u_{ik}^i,v_i^0) = 1$ and $d(u_{ij}^i,u_{ik}^i)\leq2.$ Since $d(v_x^0,v_i^0)=s, \qquad 1\leq s\leq n-1,$ and $ d(u_{ij}^i,v_i^0)=1,$ then $d(v_x^0,u_{ik}^i) = d(v_x^0,v_i^0)+d(v_i^0,u_{ik}^i) =s+1.$ Hence, $ d(u_{ij}^i,u_{ik}^i) < d(v_x^0,u_{ik}^i),$ and therefore $ r(v_x^0\mid W)\neq r(u_{ij}^i\mid W).$

    \item[(3c-iv)] $u_{xj}^x\in V(H_x^x)$ and $u_{yj}^y\in V(H_y^y)$, where $x\neq y$.

    Since $d(u_{xj}^x,v_x^0)=1, \qquad d(v_x^0,v_y^0)=s,\qquad 1\leq s\leq n-1$. Suppose $u_{yk}^y\in W$, We have $d(u_{yj}^y,u_{yk}^y)\leq 2$ and $d(v_y^0,u_{yj}^y)=d(v_y^0,u_{yj}^y)=1$. We know that $d(u_{xj}^x,v_y^0) = d(u_{xj}^x,v_x^0) + d(v_x^0,v_y^0) = 1+s$ and $d(u_{xj}^x,u_{yk}^y) = d(u_{xj}^x,v_y^0) + d(v_y^0,u_{yk}^y) = (1+s)+1 = 2+s.$ Thus, $d(u_{yj}^y,u_{yk}^y) < d(u_{xj}^x,u_{yk}^y).$ Hence, for $u_{xj}^x,u_{yj}^y \in V(G\odot\mathcal{H})\setminus W,$ it follows that $r(u_{xj}^x\mid W) \neq r(u_{yj}^y\mid W).$
\end{enumerate}

Next, to prove that $W$ is a diametral resolving set with minimum cardinality, suppose $T\subseteq V(G\odot\mathcal{H})$ contains the diametral set with $|T|<|W|.$ Let $|T|=|W|-1.$ Then there is an index $i$ such that at most $|B_i^i|-1$ vertices in $K_1+H_i^i$ are elements of $T$. Consequently, there are two vertices $u_{ij}^i,u_{ik}^i\in V(H_i^i)$, $u_{ij}^i,u_{ik}^i\notin T$, such that $r(u_{ij}^i\mid B_i^i) = r(u_{ik}^i\mid B_i^i).$ Moreover, for every vertex $z\in V(P_n\odot\mathcal H)\setminus V(H_i^i)$, every path from $z$ to a vertex of $H_i^i$ must pass through $v_i^0$. Since both $u_{ij}^i$ and $u_{ik}^i$ are adjacent to $v_i^0$, we have $d(z,u_{ij}^i) =d(z,v_i^0)+1=d(z,u_{ik}^i).$ Therefore, no vertex outside $H_i^i$ can distinguish $u_{ij}^i$ and $u_{ik}^i$. Consequently, $T$ is not a diametral resolving set. By Lemma~\ref{lemma:minimum}, $W$ is the diametral resolving set with minimum cardinality. Thus, $\dim_{\mathrm{diam}}(G\odot\mathcal{H}) = |D(G\odot\mathcal{H})| + \sum_{i,v_i\notin D(G)}\dim(K_1+H^i).$ 
\end{enumerate} 
\end{proof}

As a consequence of Theorem~\ref{thm:3.6}, we obtain the following corollary.
\begin{corollary}
The value of $\dim_{\mathrm{diam}}(G\odot\mathcal{H})$ does not depend on the position of $H^i$ for every $v_i\notin D(G)$.
\end{corollary}

\section{Conclusions}
The concept of diametral metric dimension is introduced to extend the notion of metric dimension by requiring a resolving set to contain all diametral vertices of the graph. This study investigated the diametral metric dimension of the generalized corona graph $G\odot\mathcal{H}$, where $G$ is a connected graph and $\mathcal{H}=(H^1,H^2,\ldots,H^n)$ is a sequence of connected graphs. The results show that the diametral metric dimension is influenced by the diameter of the main graph $G$, and the metric dimensions of the attached graphs $H^i$ or $K_1+H^i$, depends on graph $H^i$ contains the dominant vertex. The diametral set identifies the vertices that are most distant in the network, while the attached graphs contribute to the resolving requirements of the resulting generalized corona graph. In the context of network planning, these findings provide theoretical insights into how different local structures attached to the vertices of a main network can affect the identification and resolution of remote locations. Further studies could explore other families of graphs in the sequence $\mathcal{H}$ and investigate the diametral metric dimension under other graph operations.

\subsection*{Acknowledgement}
The research was supported by Airlangga Research Fund (ARF) year 2026 Contract No. 1819/B/DST/UN3.DRI/PT.01.03/2026.

\vspace{0.3cm}

\end{document}